\documentclass[preprint,12pt]{elsarticle}

\usepackage{amssymb}
\usepackage{amsthm}
\usepackage{amsmath}

\journal{Applied Mathematics and Computation}

\begin{document}
	\newtheorem{lem}{Lemma}
	\newtheorem{thm}{Theorem}
	\newtheorem{ex}{Example}
	\begin{frontmatter}
		
		
		
		\title{Two  new lower bounds for the smallest singular value}
		
		
		\author[inst1]{Shun Xu\footnote{Email: xushun@mail.ustc.edu.cn}}
		
		\affiliation[inst1]{organization={School of Mathematical Sciences},
			addressline={University of Science and Technology of China}, 
			city={Hefei},
			postcode={230026}, 
			%state={State One},
			country={China}}

		\begin{abstract}
			In this paper, we obtain two new lower bounds for the smallest singular value of nonsingular matrices which is better than the bound presented by Zou \cite{zou2012lower}, Lin  and Xie  \cite{lin2021some} under certain circumstances.
		\end{abstract}

		
	\end{frontmatter}
	
	
	
	\section{Introduction}
	Let $M_{n}(n\geqslant2)$ be the space of $n \times n$ complex matrices. Let $\sigma_{i}$ $(i=1, \cdots, n)$ be the singular values of $A \in M_{n}$ which is nonsingular and suppose that $\sigma_{1} \geqslant \sigma_{2} \geqslant \cdots \geqslant \sigma_{n-1} \geqslant \sigma_{n} > 0$. For
	$A=\left[a_{i j}\right] \in M_{n}$, the Frobenius norm of $A$ is defined by
	$$
	\|A\|_{F}=\left(\sum_{i, j=1}^{n}\left|a_{i j}\right|^{2}\right)^{1 / 2}={\rm tr}\left( A^{H} A\right)^{\frac{1}{2}}
	$$
	where $A^{H}$ is the conjugate transpose of $A$. The relationship between the Frobenius norm and singular values is
	$$
	\|A\|_{F}^{2}=\sigma_{1}^{2}+\sigma_{2}^{2}+\cdots+\sigma_{n}^{2}
	$$
	It is well known that lower bounds for the smallest singular value $\sigma_{n}$ of a nonsingular matrix $A \in M_{n}$ have many potential theoretical and practical applications \cite{1985Matrix,horn1994topics}. Yu and Gu \cite{yi1997note} obtained a lower bound for $\sigma_{n}$ as follows:
	$$
	\sigma_{n} \geqslant |\det A| \cdot\left(\frac{n-1}{\|A\|_{F}^{2}}\right)^{(n-1) / 2}=l>0
	$$
	The above inequality is also shown in \cite{piazza2002upper}.  In \cite{zou2012lower}, Zou improved the above inequality by showing that
	$$
	\sigma_{n } \geqslant|{\det} A|\left(\frac{n-1}{\|A\|_{F}^{2}-l^2}\right)^{(n-1) / 2}=l_0
	$$
	In \cite{lin2021some}, Lin, Minghua and Xie, Mengyan improve a lower bound for
	smallest singular value of matrices by showing that $a$ is the smallest positive solution to the equation
	$$
	x^{2}\left(\|A\|_{F}^{2}-x^{2}\right)^{n-1}=|{\det} A|^{2}(n-1)^{n-1} .
	$$
	and $\sigma\geqslant a>l_0$.
	
	In this paper, we obtain  two new lower bounds for the smallest singular value of nonsingular matrices. We give some numerical examples which will show that our result is better than $l_0$ and $a$ under certain circumstances.
	\section{Main results}
	\begin{lem}\label{lem1}
		Let\[
		l_0=|{\det} A|\left(\frac{n-1}{\|A\|_{F}^{2}-l^2}\right)^{(n-1) / 2}
		\]
		then $\sigma_n>l_0$.
	\end{lem}
	
	\begin{proof}
		In \cite{zou2012lower}, we have
		\[
		\sigma_{n } \geqslant|{\det} A|\left(\frac{n-1}{\|A\|_{F}^{2}-\sigma_n^2}\right)^{(n-1) / 2}
		\]
		since $\sigma_n\geqslant l_0>l$, thus
		\[
		\begin{aligned}
			\sigma&\geqslant|{\det} A|\left(\frac{n-1}{\|A\|_{F}^{2}-\sigma_n^2}\right)^{(n-1) / 2}\\&\geqslant|{\det} A|\left(\frac{n-1}{\|A\|_{F}^{2}-l_0^2}\right)^{(n-1) / 2}\\&>|{\det} A|\left(\frac{n-1}{\|A\|_{F}^{2}-l^2}\right)^{(n-1) / 2}=l_0
		\end{aligned}
		\]
		so $\sigma_n>l_0$.
	\end{proof}
	
	\begin{thm}
		Let $A \in M_{n}$ be nonsingular. Then
		\[
		\left(l_0^2+|\det(l_0^2 I_n-A^HA)|\left(\frac{n-1}{\|A\|_F^2-n l_0^2}\right)^{n-1}\right)^{1/2}=l_1
		\]
		then $\sigma_n\geqslant l_1$,
		where 
		\[
		l=\left|\det A\right|\left(\frac{n-1}{\|A\|_{F}^{2}}\right)^{\frac{n{-1}}{2}}
		,
		l_0=\left|\det A\right|\left(\frac{n-1}{\|A\|_{F}^{2}-l^2}\right)^{\frac{n{-1}}{2}}
		\]
	\end{thm}
	
	\begin{proof}
		Let $0<\lambda<\sigma_n^2$, denote
		\[
		\left|\left(\lambda-\sigma_{1}^{2}\right)\left(\lambda-\sigma_{2}^{2}\right) \cdots\left(\lambda-\sigma_{n-1}^2\right)\right|\leqslant\left(\frac{\sigma_1^2 +\cdots+\sigma_{n-1}^{2}-(n-1) \lambda}{n-1}\right)^{n-1}
		\]
		Since
		\begin{equation*}
			\begin{aligned}
				\left|\left(\lambda-\sigma_{1}^{2}\right)\left(\lambda-\sigma_{2}^{2}\right) \cdots\left(\lambda-\sigma_{n-1}^2\right)\right|
				&=\frac{\left|\left(\lambda-\sigma_{1}^{2}\right)\left(\lambda-\sigma_{2}^{2}\right) \cdots\left(\lambda-\sigma_{n}^2\right)\right|}{\sigma_n^2-\lambda}\\
				&=\frac{|\det(\lambda I_n-A^HA)|}{\sigma_n^2-\lambda}
			\end{aligned}
		\end{equation*}
		then
		\[
		\frac{|\det(\lambda I_n-A^HA)|}{\sigma_n^2-\lambda}\leqslant\left(\frac{\sigma_1^2 +\cdots+\sigma_{n-1}^{2}-(n-1) \lambda}{n-1}\right)^{n-1}
		\]
		\[
		\sigma_n^2\geqslant\lambda+|\det(\lambda I_n-A^HA)|\left(\frac{n-1}{\sigma_1^2 +\cdots+\sigma_{n-1}^{2}-(n-1) \lambda}\right)^{n-1}
		\]
		\[
		\sigma_n\geqslant \left(\lambda+|\det(\lambda I_n-A^HA)|\left(\frac{n-1}{\sigma_1^2 +\cdots+\sigma_{n-1}^{2}-(n-1) \lambda}\right)^{n-1}\right)^{1/2}
		\]
		By Lemma \ref{lem1}, $l_0<\sigma_n$, $l_0^2< \sigma_n^2$, let $\lambda=l_0^2$, then
		\begin{equation}\label{sigma}
			\sigma_n\geqslant \left(l_0^2+|\det(l_0^2 I_n-A^HA)|\left(\frac{n-1}{\|A\|_F^2-\sigma_n^2-(n-1) l_0^2}\right)^{n-1}\right)^{1/2}
		\end{equation}
		Therefore
		\[
		\sigma_n\geqslant \left(l_0^2+|\det(l_0^2 I_n-A^HA)|\left(\frac{n-1}{\|A\|_F^2-n l_0^2}\right)^{n-1}\right)^{1/2}
		\]
	\end{proof}

	\begin{thm}\label{thm2}
		Let $A \in M_{n}$ be nonsingular. Let
		\[
		b_{k+1}= \left(l_0^2+|\det(l_0^2 I_n-A^HA)|\left(\frac{n-1}{\|A\|_F^2-(n-1) l_0^2-b_{k}^2}\right)^{n-1}\right)^{1/2},k=1,2,\cdots
		\]
		with $
		l=\left|\det A\right|\left(\frac{n-1}{\|A\|_{F}^{2}}\right)^{\frac{n{-1}}{2}}
		,
		l_0=\left|\det A\right|\left(\frac{n-1}{\|A\|_{F}^{2}-l^2}\right)^{\frac{n{-1}}{2}}
		$
		\[
		b_1=\left(l_0^2+|\det(l_0^2 I_n-A^HA)|\left(\frac{n-1}{\|A\|_F^2-(n-1) l_0^2}\right)^{n-1}\right)^{1/2}
		\]
		then $0<b_k<b_{k+1}\leqslant \sigma_n,k=1,2,\cdots$, $\lim_{k\to\infty}b_k$ exists.
	\end{thm}
	\begin{proof}
		We show by induction on $k$ that
		\[
		\sigma_n\geqslant b_{k+1}>b_k>0
		\]
		By (\ref{sigma}), we have
		\[
		\begin{aligned}
			\sigma_n&\geqslant \left(l_0^2+|\det(l_0^2 I_n-A^HA)|\left(\frac{n-1}{\|A\|_F^2-\sigma_n^2-(n-1) l_0^2}\right)^{n-1}\right)^{1/2}\\
			&\geqslant \left(l_0^2+|\det(l_0^2 I_n-A^HA)|\left(\frac{n-1}{\|A\|_F^2-(n-1) l_0^2}\right)^{n-1}\right)^{1/2}=b_1
		\end{aligned}
		\]
		so $\sigma_n\geqslant b_1$, then
		\[
		\begin{aligned}
			\sigma_n&\geqslant \left(l_0^2+|\det(l_0^2 I_n-A^HA)|\left(\frac{n-1}{\|A\|_F^2-\sigma_n^2-(n-1) l_0^2}\right)^{n-1}\right)^{1/2}\\
			&\geqslant \left(l_0^2+|\det(l_0^2 I_n-A^HA)|\left(\frac{n-1}{\|A\|_F^2-(n-1) l_0^2-b_1^2}\right)^{n-1}\right)^{1/2}=b_2\\
			&> \left(l_0^2+|\det(l_0^2 I_n-A^HA)|\left(\frac{n-1}{\|A\|_F^2-(n-1) l_0^2}\right)^{n-1}\right)^{1/2}=b_1>0
		\end{aligned}
		\]
		When $k=1$, we have
		\[
		\sigma_n\geqslant b_{2}>b_1>0
		\]
		Assume that our claim is true for $k=m$, that is $\sigma_{n } \geqslant b_{m+1} > b_{m}>0 .$ Now we consider the case when $k=m+1$. By (\ref{sigma}), we have
		\[
		\begin{aligned}
			\sigma_n&\geqslant \left(l_0^2+|\det(l_0^2 I_n-A^HA)|\left(\frac{n-1}{\|A\|_F^2-\sigma_n^2-(n-1) l_0^2}\right)^{n-1}\right)^{1/2}\\
			&\geqslant \left(l_0^2+|\det(l_0^2 I_n-A^HA)|\left(\frac{n-1}{\|A\|_F^2-b_{m+1}^2-(n-1) l_0^2}\right)^{n-1}\right)^{1/2}=b_{m+2}\\
			&> \left(l_0^2+|\det(l_0^2 I_n-A^HA)|\left(\frac{n-1}{\|A\|_F^2-b_m^2-(n-1) l_0^2}\right)^{n-1}\right)^{1/2}=b_{m+1}>0
		\end{aligned}
		\]
		Hence $\sigma_n\geqslant b_{m+2}>b_{m+1}>0$. This proves $\sigma_n\geqslant b_{k+1}>b_{k}>0,k=1,2,\cdots$. By the well known monotone convergence theorem, $\lim_{k\to \infty}b_k$ exists.
	\end{proof}
	
	\begin{thm}\label{thm3}
		Let $b=\lim _{k \rightarrow \infty} b_{k}$,
		\[
		f(x)=\left(l_0^2+|\det(l_0^2 I_n-A^HA)|\left(\frac{n-1}{\|A\|_F^2-x^2-(n-1) l_0^2}\right)^{n-1}\right)^{1/2}
		\]then $b$ is the smallest positive solution to the equation $x=f(x)$,and $\sigma_n\geqslant b$.
	\end{thm}
	\begin{proof}
		Let $x_0$ is the smallest positive solution to the equation $x=f(x)$, we show by induction on $k$ that $x_0>b_k,k=1,2,\cdots$. When $k=1$
		\[
		\begin{aligned}
			x_0&=\left(l_0^2+|\det(l_0^2 I_n-A^HA)|\left(\frac{n-1}{\|A\|_F^2-x_0^2-(n-1) l_0^2}\right)^{n-1}\right)^{1/2}\\
			&>\left(l_0^2+|\det(l_0^2 I_n-A^HA)|\left(\frac{n-1}{\|A\|_F^2-(n-1) l_0^2}\right)^{n-1}\right)^{1/2}=b_1\\
		\end{aligned}
		\]
		Assume that our claim is true for $k = m$, that is $\sigma_n>b_m$. Now we consider the case when $k = m + 1$. 
		\[
		\begin{aligned}
			x_0&=\left(l_0^2+|\det(l_0^2 I_n-A^HA)|\left(\frac{n-1}{\|A\|_F^2-x_0^2-(n-1) l_0^2}\right)^{n-1}\right)^{1/2}\\
			&>\left(l_0^2+|\det(l_0^2 I_n-A^HA)|\left(\frac{n-1}{\|A\|_F^2-b_m^2-(n-1) l_0^2}\right)^{n-1}\right)^{1/2}=b_{m+1}\\
		\end{aligned}
		\]
		Hence $x_0>b_{m+1}$. This proves $x_0>b_k,k=1,2,\cdots$. Since $b$ is a positive solution to the equation $x=f(x)$ and $x_0>b_k,k=1,2,\cdots$, then $b=x_0$. Therefore $b$ is the smallest positive solution to the equation $x=f(x)$ and $\sigma_n\geqslant b$.
	\end{proof}
	Therefore we obtain  two new lower bounds $l_1$ and $b$ for the smallest singular value of nonsingular matrices.
	
	\section{Numerical examples}
	We use Examples \ref{ex1} and Example \ref{ex2} to compare the values of $l, l_0, l_1$.
	\begin{ex}\label{ex1}
		Let
		\[
		A=\left[\begin{array}{ccc}
			4 & -4 & -3 \\
			3 & 4 & 2 \\
			4 & 1 & 0
		\end{array}\right]
		\]
	\end{ex}
	Then $\sigma_{\min }=0.0231$, and
	\[
	l=0.0229885
	\]
	\[
	l_0=0.0229886
	\]
	Our result:
	\[
	l_1=0.0230691
	\]
	
	\begin{ex}\label{ex2}
		Let
		\[
		A=\left[\begin{array}{ccc}
			4 & 0 & 0 \\
			-1 & 5 & 0 \\
			0 & 5 & 4
		\end{array}\right]
		\]
	\end{ex}
	Then
	\[
	l=1.92771
	\]
	\[
	l_0=2.01806
	\]
	Our result:
	\[
	l_1=2.31515
	\]
	
	Next we use the following example to compare the values of $a,b,l_1$.
	\begin{ex}
		Let
		\[
		A=\left[\begin{array}{lll}
			3 & 2 & 0 \\
			1 & 9 & 5 \\
			0 & 5 & 7
		\end{array}\right]
		\]
	\end{ex}
	Then
	\[
	a=1.0367
	\]
	Our result:
	\[
	l_1=1.3434
	\]
	\[
	b=1.3455
	\]
	
	\section{Acknowledgments}
	I would like to express my sincere gratitude to professor Xiao-Wu Chen from University of Science and Technology of China.
	\bibliographystyle{elsarticle-num} 
	\bibliography{cas-refs}
	
	
	
	
	
	% \end{thebibliography}
\end{document}